\documentclass[12pt]{article}
\usepackage{amsfonts}
\usepackage{wasysym,amssymb,eufrak,indentfirst,graphicx,cite,amsthm,color}
\usepackage[bookmarksnumbered, colorlinks, plainpages]{hyperref}
\newtheorem{Theorem}{Theorem}[section]
\newtheorem{Corollary}[Theorem]{Corollary}
\newtheorem{Lemma}[Theorem]{Lemma}
\newtheorem{Proposition}[Theorem]{Proposition}
\newtheorem{Question}[Theorem]{Question}
\newtheorem{Definition}[Theorem]{Definition}
\newtheorem{Example}[Theorem]{Example}
\newtheorem{Remark}{{Remark}}

\def\qed{\hfill $\Box$}

\begin{document}
\baselineskip 17pt

\title{Weakly nilpotent hypergroups \thanks{Research was supported by the NNSF  of China (12371025, 12371102).}}

\author{Chi Zhang, Jun Liu\thanks{Corresponding author}, Dengyin Wang\\
{\small  School of Mathematics, JCAM, China University of Mining and
Technology}\\
{\small Xuzhou, 221116, P. R. China}\\
{\small E-mails: zclqq32@cumt.edu.cn; junliu@cumt.edu.cn; wdengyin@126.com}\\}

\date{}
\maketitle

\begin{abstract}

In this paper, we introduce the weakly nilpotent hypergroups with giving
some new properties, and then establish several structural characterizations of these hypergroups.
Some results obtained in this paper answer the two questions raised by the first author and W. Guo in \cite{zhang1}.
\end{abstract}

\let\thefootnoteorig\thefootnote
\renewcommand{\thefootnote}{\empty}


 \let\thefootnote\thefootnoteorig

\section{Introduction}
F. Marty \cite{m1} introduced the notion of hypergroups which generalizes the definition of groups.
The following definitions regarding hypergroups are due to P.H.-Zieschang and various co-authors \cite{b1,fz1,vz1}.

\begin{Definition}

A hypergroup is a set $H$ equipped with a hypermultiplication (a map from $H$ to the power set of $H$, denoted as $(p, q)  \longmapsto pq$ for all $p, q \in H$).
For any subsets $P, Q\subseteq H$,
$$PQ := \bigcup_{p \in P, q \in Q} pq.$$
If $P = \{ p \}$, a singleton set, then $pQ := \{ p \}Q$, and $Qp := Q \{ p \}$.
The hypermultiplication is assumed to satisfy the following conditions:

$(H1)$  For any elements $p, q$, and $r$ in $H$, $p(qr) = (pq)r$.

$(H2)$ $H$ contains an element $1$ such that $s1 = \{ s \}$ for all $s \in H$.

$(H3)$ For each element $s$ in $H$, there exists an element $s^*$ in $H$ such that for any elements $p, q$, and $r$ in $H$ with $r \in pq$, then we have $q \in p^*r$ and $p \in rq^*$.

\end{Definition}

$(H1)$ implies that set product is associative. $(H2)$ and $(H3)$ yield that for all $s \in H$, $1 \in s^*s$.
An element $s \in H$ is called {\em thin} if $s^*s = \{ 1 \}$.

The set of all the thin elements of a hypergroup $H$ is denoted $O_{\vartheta}(H)$.
A hypergroup $H$ is called thin if $H = O_{\vartheta}(H)$, that is, all elements of $G$ are thin.
Of course, any group $H$ may be regarded as a thin hypergroup, simply by replacing the product of two elements with the singleton set containing that product.
Conversely, it is easy to see that a thin hypergroup is a group.
In fact, the concept of hypergroups is a generalization of groups.

Later, a number of group theoretic results have found a generalization within the theory of hypergroups (see \cite{b1, fz1, fz2, tz1, vz1, z2, z4, zhang1}).
We need to specifically mention here that P.-H. Zieschang has published a monograph \cite{z3} on hypergroups.
In 2022, A. V. Vasil'ev and P.-H. Zieschang \cite{vz1} proposed the definition of solvable hypergroups: A hypergroup $H$ is said to be solvable if it contains closed subsets $F_{0}, \cdots, F_n$ such
that $F_0 = {1}$, $F_n = H$, and, for each element $i$ in $\{0, \cdots, n\}$ with $1 \leq i$, $F_{i-1} \leq F_i$,
$F_i//F_{i-1}$ is thin, and $|F_i//F_{i-1}|$ is a prime number.
This generalized the theory of finite solvable groups to solvable hypergroups.
H. Blau \cite{b2} give the definition of $\pi$-separable hypergroups,
and further  generalized the results of   A. V. Vasil'ev and P.-H. Zieschang on residually thin and $\pi$-valenced hypergroups.
Recently, C. Zhang and W. Guo \cite{zhang1} proposed the definition of nilpotent hypergroups: A hypergroup $H$ is said to be nilpotent if $H_{n} = \underbrace{[H, H, \cdots, H]}_{n} = 1$ for some positive integer $n$, where $H_{n} = [H_{n-1}, H]$, $H_{1} = H$.
In this paper, we should give the definition of weakly nilpotent hypergroups.
And we should study some properties of weakly nilpotent hypergroups and provided some structural characterizations of weakly nilpotent hypergroups.
Then we will answer two questions in \cite{zhang1}.

Following \cite{fz1, vz1},
for any subset $S$ of a hypergroup $H$, let $S^* := \{ s^* | s \in S \}$;
A nonempty subset $F$ of $H$ is called {\em closed} if $F^*F \subseteq F$, that is, $a^*b \subseteq F$ for all $a, b \in F$.
Clearly, the intersection of some closed subsets of $H$ is also a closed subset of $H$.
A closed subset $F$ of $H$ is called {\em normal} ({\em strongly normal}) in $H$ if $Fh \subseteq hF$ ($h^*Fh \subseteq F$) for each element $h$ in $H$;
It is easy to see that strong normality implies normality.
A closed subset $F$ is {\em subnormal} ({\em strongly subnormal}) in $H$ if
there exists a chain of closed subsets $F = F_0 \subseteq F_1 \subseteq \cdots \subseteq F_n = H$ for some $n > 0$ such that $F_{i-1}$ is normal (strongly normal) in $F_i$ for all $1 \leq i \leq n$.
The intersection of all closed subsets of $H$ which are strongly normal in $H$ will be denoted by $O^{\vartheta}(H)$.
The subset $O^{\vartheta}(H)$ of $H$ is called the {\em thin residue} of $H$.
Lemma $3.6(ii)$ in \cite{fz1} implies that the thin residue of $H$ is a strongly normal closed subset of $H$.
Let $F$ be a closed subset of a hypergroup $H$.
For any $h \in H$, define $h^F := FhF$; and for any subset $F \subseteq S \subseteq H$, define $S//F := \{ h^F | h \in S \}$.
Then for all $a, b \in H$, $a^F \cdot b^F := \{x^F | x \in aFb \}$
defines a hypermultiplication on $H//F$ such that $H//F$ becomes a hypergroup, called the quotient of $H$ over $F$.

Following \cite[p.100]{fz1}, a subset $A$ of $H$ is said to be a {\em star invariant} subset of $H$ if $A^{*} = A.$
For each subset $A$ of $H$,  $\langle A \rangle$ denote to be the intersection of all closed subsets of
$H$ containing $A$.

Following \cite{b2}, a hypergroup $H$ is called residually thin (or in short $RT$)
if there exists a chain of closed subsets
$\{1\} = F_0 \subset F_1 \subset \ldots \subset F_n = H$ such that $F_{i}//F_{i-1}$ is thin for all $1 \leq i \leq n$
and the valency of a finite $RT$ hypergroup $H$ is the integer
$$n_{H} = \prod_{i=1}^{n} |F_{i}//F_{i-1}|.$$
A closed subset $C$ of a residually thin hypergroup is called a $p$-subset if $n_C$ is a $p$-number;
A Sylow $p$-subset of $H$ is a $p$-subset $C$ such that $n_{H}/n_{C}$ is a $p'$-number.
An element $h \in H$ is called {\em $p$-valenced} if $n_{{h^*}^{U}h^{U}} = |{h^*}^{U}h^{U}|$ is a $p$-number
for every subnormal closed subset $U$ such that $n_U$ is a $p$-number and ${h^*}^{U}h^{U} \in O_{\vartheta}(H//U)$.
If all $h \in H$ are $p$-valenced, then $H$ is siad to be {\em $p$-valenced}.

Let $Z(H) = \{ h \in H | hx=xh, \forall x \in H, h^*h=1 \}$ and we call that $Z(H)$ is the centre of $H$.
Note that $Z(H)$ is a normal closed subset of $H$ (See Lemma \ref{normal} below).

\begin{Definition}\label{hyper}

Let $H$ be a hypergroup.
We call that a  subset series of $H$: $$1 = Z_{0}(H) \subseteq Z_{1}(H) \subseteq \cdots \subseteq Z_{n}(H) \subseteq \cdots$$  the upper center series of $H$
if $Z_{i}(H)//Z_{i-1}(H) = Z(H//Z_{i-1}(H))$ for all $i = 1, 2, \cdots$.

\end{Definition}

\begin{Remark}
$(1)$ $Z_{n}(H)$ is a normal closed subset of $H$ for all positive integer $n$ (See Lemma \ref{normal} below).

$(2)$ When $Z_{n}(H) = Z_{n+1}(H) = \cdots$, then we denote by $Z_{\infty}(H)$ the terminal term of this ascending series and say that $Z_{\infty}(H)$ is the hypercenter of $H$.

\end{Remark}

\begin{Definition}\label{nilpotent}

We call that a hypergroup $H$ is weakly nilpotent if $Z_{n}(H) = H$ for some non-negative integer $n$.

\end{Definition}

\begin{Remark}

$(1)$ $Z_{\infty}(H)$ is weakly nilpotent for every hypergroup $H$.

$(2)$ Nilpotent groups are thin weakly nilpotent hypergroups. Of course, thin weakly nilpotent hypergroups are nilpotent groups as well.

\end{Remark}

A hypergroup $H$ is called a {\em commutative hypergroup} if for every $a, b \in H$, $ab=ba$.

\begin{Example}\label{E1}(See \cite[Example 2.12]{jun})

Let $K: = \{0, 1\}$ be a two point set.
One imposes a commutative hyperoperation $+$ as follows:
$0 + 0 = 0, 1 + 0 = 1, 1 + 1 = \{0, 1\}.$

\end{Example}

Observe that $K$ is a commutative hypergroup.
However $K$ is not a weakly nilpotent hypergroup.
In fact, $K$ is not $RT$ and all weakly nilpotent hypergroups are $RT$.

Next we combine the relationship between the definitions of nilpotent hypergroups from \cite{zhang1}
and weakly nilpotent hypergroups introduced in present paper.
If $H$ is nilpotent (the definition in \cite{zhang1}), then there exists some positive integer $n$, where $H_{n} = [H_{n-1}, H]$, $H_{1} = H$.
From Lemma \ref{thin residue}(1) below,  $O^{\vartheta}(H) = [H, 1] = [1, H] \leq [H_{n-1}, H]=H_{n} =1$.
It implies from Lemma \ref{thin residue}(2) below that $1$ is a strongly normal closed subset of $H$.
By Lemma \ref{strong} below, we have $H$ is thin.
Hence $H$ is a nilpotent group.
Of course, $H$ is a weakly nilpotent hypergroup.
Conversely, it is not in general.

\begin{Example}\label{E2}

Let $H: = \{0, 1, 2\}$ be a three point set.
One imposes a commutative hyperoperation $+$ as follows:
\[
\begin{array}{c|ccc}
+ & 0 & 1 & 2 \\
\hline
0 & 0 & 1 & 2 \\
1 & 1 & 0 & 2 \\
2 & 2 & 2 & \{0,1\}.
\end{array}
\]

\end{Example}

We also get some characterization of finite weakly nilpotent hypergroups.

\begin{Theorem}\label{sub}

Closed subsets of weakly nilpotent hypergroups are weakly nilpotent.

\end{Theorem}

\begin{Theorem}\label{qu}

Let $H$ be a weakly nilpotent hypergroup, and let $T$ be a closed subset of $H$.
Then $H//T$ is a weakly nilpotent hypergroup.

\end{Theorem}

In finite group theory, Wielandt proved a well-known result: a finite group $G$ is a nilpotent group
if and only if every subgroup of $G$ is a subnormal subgroup of $G$.
In  \cite{zhang1}, Zhang and Guo  proposed an open question.

\begin{Question}

Let $H$ be a finite nilpotent hypergroup.
Is any closed subset of $H$ is  subnormal in $H$ ?

\end{Question}

In this paper, we  provide a positive answer to this question and get the following theorem.

\begin{Theorem}\label{subnormal}

Closed subsets of finite weakly hypergroups $H$ are strongly  subnormal in $H$.

\end{Theorem}

In group theory, there is a well-known result: nilpotent groups are solvable.
In finite hypergroup thoery, we obtain a similar result.

\begin{Theorem}\label{solvable}

Let $H$ be a finite weakly nilpotent hypergroup.
Assume that $H$ is residually thin and $p$-valenced for some prime $p$ divides $n_H$.
Then $H$ is solvable.

\end{Theorem}

In group theory, there is another well-known result about nilpotent groups:
if a finite group $G$ is nilpotent, then every Sylow subgroup of $G$ is a normal subgroup of $G$.
In  \cite{zhang1}, Zhang and Guo  proposed the other open question.

\begin{Question}

Let $H$ be a finite $RT$ nilpotent hypergroup.
Is any Sylow closed subset $P$ of $H$ is normal in $H$ $?$

\end{Question}

In this paper, we also answer this question with $p$-valenced and get the following theorem.

\begin{Theorem}\label{Sylow}

Let $H$ be a finite weakly nilpotent hypergroup.
Assume that $H$ is residually thin and $p$-valenced for some prime $p$ divides $n_H$.
Then every Sylow $p$-subset $P$ of $H$ is strongly normal in $H$.

\end{Theorem}

\begin{Remark}

It is clear that strongly normality is normality.
Hence we answer this question with $p$-valenced.
We do not know if the conditional conclusion of removing $p$-valenced holds true.

\end{Remark}

This paper is organized as follows. In Section $2$, we cite some known results which are useful
in our proofs and prove some basic properties of hypergroups.
In Sections $3$, we give the proofs of Theorems, \ref{sub}, \ref{qu} and \ref{subnormal}.
In Sections $4$, we give the proofs of Theorems \ref{solvable} and \ref{Sylow}.
In Section $5$, we further present some propositions and questions on weakly nilpotent hypergroups.

In this paper, the letter $H$ stands for a  hypergroup.

\section{Some known results and basic preliminaries}

\begin{Lemma}\label{closed}

$(1)$  A subset A of $H$ is closed if and only if $1 \in A$, $A^{*} \subseteq A$ and $AA \subseteq A$. (see \cite[Lemma 2.1.1]{z3})

$(2)$ Let $A, B$ be subsets of $H$. Then $(AB)^* = B^*A^*$.

\end{Lemma}

\begin{Lemma}\label{thin residue}

Let $H$ be a hypergroup.

(1) $O^{\vartheta}(H) = [H, {1}]$ (See \cite[Lemma 5.4]{fz1}).

(2) $O^{\vartheta}(H)$ is strongly normal in $H$ (See \cite[Lemma 3.6(ii)]{fz1}).

\end{Lemma}

\begin{proof}
Since $N$ is normal in $H$, for every $f \in F$, ${(h^N)}^{*}{f^N}{h^N} \in F//N$ if and only if
$N{h^{*}}fhN \subseteq F$ if and only if ${h^{*}}fh \in F$.
Hence $F//N$ is strongly normal in $H//N$ if and only if $F$ is strongly normal in $H$.
\end{proof}

\begin{Lemma}\label{strong} {\rm (See \cite[Lemma 3.7]{vz1})}

A closed subset $F$ of $H$ is strongly normal in $H$ if and only if $H//F$ is thin.

\end{Lemma}

\begin{Lemma}\label{a}

Let $p$ be a prime number, let $H$ be a $p$-valenced hypergroup, and
let $F$ be a strongly normal closed subset of $H$.
Then $F$ is $p$-valenced.
\end{Lemma}

\begin{proof}
For every element $f \in F$,
and for every subnormal closed subset $U$ of $F$ such that $n_U$ is a $p$-number and ${f^*}^{U}f^{U} \in O_{\vartheta}(F//U)$.
Since $F$ be a strongly normal closed subset of $H$,
we have that $U$ is a subnormal closed subset of $H$.
And ${f^*}^{U}f^{U} \in O_{\vartheta}(F//U) \subseteq O_{\vartheta}(H//U).$
It follows from  $H$ is $p$-valenced that $n_{{f^*}^{U}f^{U}} = |{f^*}^{U}f^{U}|$ is a $p$-number.
Hence $F$ is $p$-valenced.
\end{proof}

\begin{Lemma}\label{normal}
$Z_{n}(H)$ is a normal closed subset of $H$ for all positive integer $n$.
\end{Lemma}

\begin{proof}
For $n=1$, $Z_{1}(H) = Z(H)$ is a normal closed subset.
In fact, $1 \in Z(H)$.
For every element $y \in Z(H)Z(H)$,
then there exist two $a, b \in Z(H)$,  $y in ab$.
By Lemma 1.4.3(i) in \cite{z3}, $|ab|=1$ and so $y = ab$.
For every element $h$ in $H$, $(ab)h = a(hb) =  h(ab)$ ,that is, $yh = hy$.
Next $y^*y = {(ab)}^{*}(ab) = b^*a^*ab =1$.
Hence $Z(H)Z(H) \subseteq Z(H)$.
By Lemma \ref{closed}(2), we have $a^*h = {(h^*a)}^{*} = {(ah^*)}^{*} = ha^*.$
And ${(a^*)}^{*}(a^*) = aa^* = a^*a = 1$.
Then ${Z(H)}^{*} \subseteq Z(H)$.
It implies from Lemma \ref{closed}(1) that $Z_{1}(H) = Z(H)$ is a closed subset of $H$.
Since  $Z(H)h = hZ(H)$, $Z(H)$ is a normal closed subset.
We assume that $Z_{i}(H)$ is a normal colsed subset of $H$.
We show that $Z_{i+1}(H)$ is a normal closed subset of $H$.
Firstly, we show $Z_{i+1}(H)$ is a closed subset of $H$.
Clearly, $1 \in Z_{i+1}(H)$.
For every element $d \in Z_{i+1}(H)Z_{i+1}(H)$,
then there exist two $a, b \in Z_{i+1}(H)$,  $d \in ab$.
Then as above we have
$$d^{Z_{i}(H)} \in {(ab)}^{Z_{i}(H)} \subseteq a^{Z_{i}(H)}b^{Z_{i}(H)} \subseteq Z(H//Z_{i}(H)) = Z_{i+1}(H)//Z_{i}(H).$$
Therefore $d \in Z_{i+1}(H)$, that is, $Z_{i+1}(H)Z_{i+1}(H) \subseteq Z_{i+1}(H)$.
And For every element $x \in Z_{i+1}(H)$,
$$(x^*)^{Z_{i}(H)} = (x^{Z_{i}(H)})^{*} \in Z(H//Z_{i}(H)) = Z_{i+1}(H)//Z_{i}(H).$$
It implies that  $x^* \in Z_{i+1}(H)$.
Hence $Z_{i+1}(H)$ is a closed subset of $H$.
Finally, we show that $Z_{i+1}(H)$ is  normal in $H$.
In fact, as $Z_{i+1}(H)//Z_{i}(H) = Z(H//Z_{i}(H))$,
from the above proof we see that $Z_{i+1}(H)//Z_{i}(H)$ is normal in $H//Z_{i}(H)$.
Now by Lemma 2.7 in \cite{zhang1}, $Z_{i+1}(H)$ is normal in $H$.
Thus, we get that $Z_{n}(H)$ is a normal closed subset of $H$ for all positive integer $n$.
\end{proof}

\section{Proofs of Theorems  \ref{sub}, \ref{qu} and \ref{subnormal}}

In this section, we will prove Theorems \ref{sub}, \ref{qu} and \ref{subnormal}.
To prove our theorems, we first proved the following lemmas, which are the key steps in proving the theorems \ref{sub} and \ref{qu}.

\begin{Lemma}\label{element}

Let $x$ be a central element of $H$ and $T$ a closed subset of $H$.
Then $x^T$ is a central element of $H//T$, that is, $Z(H)T//T \subseteq Z(H//T)$.

\end{Lemma}

\begin{proof}
For every $x \in Z(H)$ and $h \in H$, we have that

$(TxT)(ThT) = TxThT =TxhT$ and $(ThT)(TxT) = ThTxT =ThxT$.
Since $xh=hx$, $$(TxT)(ThT)= TxhT = ThxT = (ThT)(TxT).$$
And
$$(Tx^{*}T)(TxT)= Tx^{*}xT = T.$$
It implies that $x^{T} \in Z(H//T)$.
Therefore $Z(H)T//T \subseteq Z(H//T)$.
\end{proof}

\begin{Definition}

We call a series of closed subsets of $H$:

$$H = H_0 \supseteq H_1 \supseteq H_2 \supseteq \cdots \supseteq H_{r-1} \supseteq H_r = 1,$$

a central series of $H$ if $ H_{i-1}//H_{i} \subseteq Z(H//H_{i})$ for $i = 1, 2, \cdots r$.

\end{Definition}

\begin{Lemma}\label{central}

Let $H$ be a hypergroup, and let $1=T_0 \subseteq T_1 \subseteq \cdots \subseteq T_{n-1} \subseteq T_{n}=H$ be a central series of $H$.
Then $T_i \subseteq Z_i(H)$ for each element $i$ in $\{0, 1,\cdots n \}$.
\end{Lemma}

\begin{proof}

Since $T_0 \subseteq T_1 \subseteq \cdots \subseteq T_{n-1} \subseteq T_{n}=H$ is a central series of $H$,
we have $T_0 = \{ 1 \} = Z_0(H)$.
Thus, if $i = 0, T_i \subseteq Z_i(H)$, so that we are done in this case.

Assume that $1 \leq i$ and that the statement is true for $i-1$.
Then $T_{i-1} \subseteq Z_{i-1}(H)$.
Thus, by \cite[Theorem 3.4.6]{z4}, $Z_{i-1}(H)//T_{i-1}$ is a closed subset of $H//T_{i-1}$,
so that we may apply to $H//T_{i-1}$ and  $Z_{i-1}(H)//T_{i-1}$ in place of $H$ and $T$.
We obtain that

$$ Z(H//T_{i-1})Z_{i-1}(H)//T_{i-1}//Z_{i-1}(H)//T_{i-1} \subseteq  Z((H//T_{i-1})//(Z_{i-1}(H)//T_{i-1})).$$

Now recall that $1=T_0 \subseteq T_1 \subseteq \cdots \subseteq T_{n-1} \subseteq T_{n}=H$ is assume to be a central series of $H$.
Thus
$$T_i//T_{i-1} \subseteq Z(H//T_{i-1}).$$
It follows that
$$ (T_i//T_{i-1})(Z_{i-1}(H)//T_{i-1})//Z_{i-1}(H)//T_{i-1} \subseteq  Z((H//T_{i-1})//(Z_{i-1}(H)//T_{i-1})).$$
From \cite[Theorem 3.7.2]{z4} we also obtain that
$$  (T_iZ_{i-1}(H))//T_{i-1} \subseteq (T_i//T_{i-1})(Z_{i-1}(H)//T_{i-1}).$$
Thus
$$ (T_iZ_{i-1}(H))//T_{i-1}//(Z_{i-1}(H)//T_{i-1}) \subseteq  Z((H//T_{i-1})//(Z_{i-1}(H)//T_{i-1}))$$
and then, by \cite[Theorem 3.7.2]{z4},
$$(T_iZ_{i-1}(H))//Z_{i-1}(H) \subseteq Z(H//Z_{i-1}(H)) = Z_{i}(H)//Z_{i-1}(H).$$
It follows that $T_i \subseteq Z_i(H)$.
\end{proof}

\begin{Corollary}
Hypergroups which posses a central series are weakly nilpotent.
\end{Corollary}

\begin{proof}
This follows immediately from Lemma \ref{central}.
\end{proof}

\textbf{Proof of Theorem \ref{sub}}~~

Since $H$ is a weakly nilpotent hypergroup, there exists a central series of $H$:
$$H = H_0 \supseteq H_1 \supseteq H_2 \supseteq \cdots \supseteq H_{r-1} \supseteq H_r = 1.$$
Let $N_i = N \cap H_i$.  Then we show that
$$N = N_0 \supseteq N_1 \supseteq N_2 \supseteq \cdots \supseteq N_{r-1} \supseteq N_r = 1$$
is a central series of $N$.
Firstly, we have that
$$ N//N_{i+1} \cong N//(N \cap H_{i+1}) \cong H_{i+1}N//H_{i+1} \subseteq H//H_{i+1}.$$
This induces that
$$ N_i//N_{i+1} \cong H_{i+1}N_i//H_{i+1} \subseteq H_i//H_{i+1}.$$
Then $N_i//N_{i+1} \subseteq Z(N//N_{i+1})$ since $H_i//H_{i+1} \subseteq Z(H//H_{i+1})$.
Therefore $N$ is a weakly nilpotent hypergroup.
\qed
\medskip

\textbf{Proof of Theorem \ref{qu}}~~

Since $H$ is a weakly nilpotent hypergroup, there exists the upper central series:
$$1 \subseteq Z_{1}(H) \subseteq Z_{2}(H) \subseteq \cdots \subseteq Z_{n-1}(H) \subseteq Z_{n}(H) = H$$
such that $Z_{i}(H)//Z_{i-1}(H) = Z(H//Z_{i-1}(H))$ for all $i = 1, 2, \cdots n$.
Then we have
$$T//T \subseteq Z_{1}(H)T//T \subseteq Z_{2}(H)T//T \subseteq \cdots \subseteq Z_{n-1}(H)T//T \subseteq Z_{n}(H)//T = H//T.$$
Since  $Z_{i}(H)//Z_{i-1}(H) = Z(H//Z_{i-1}(H))$ and Lemma \ref{element},
we have that
$$(Z_{i}(H)//Z_{i-1}(H))(Z_{i-1}(H)T//Z_{i-1}(H))//(Z_{i-1}(H)T//Z_{i-1}(H))$$
 $$= Z(H//Z_{i-1}(H))(Z_{i-1}(H)T//Z_{i-1}(H))//(Z_{i-1}(H)T//Z_{i-1}(H))$$
$$ \subseteq Z((H//Z_{i-1}(H))//(Z_{i-1}(H)T//Z_{i-1}(H))).$$
Note that
$$Z_{i}(H)T//Z_{i-1}(H) \subseteq (Z_{i}(H)//Z_{i-1}(H))(Z_{i-1}(H)T//Z_{i-1}(H)).$$
Then
$$(Z_{i}(H)T//Z_{i-1}(H))//(Z_{i-1}(H)T//Z_{i-1}(H)) \subseteq Z(H//Z_{i-1}(H))//(Z_{i-1}(H)T//Z_{i-1}(H)).$$
Thus, by \cite[Theorem 3.7.2]{z4},
$$Z_{i}(H)T//Z_{i-1}(H)T \subseteq Z(H//Z_{i-1}(H)T).$$
By \cite[Theorem 3.7.2]{z4} again,
$$(Z_{i}(H)T//T)//(Z_{i-1}(H)T//T) \subseteq Z((H//T)//(Z_{i-1}(H)T//T)).$$
By Lemma \ref{central}, $H//T$ is a weakly nilpotent hypergroup.
\qed
\medskip

Recall a closed subset $M$ is {\em maximal} in $H$ if $M \neq H$ and when there exits a closed subset $K$ of $H$ such that $M \subseteq K \subset H$, then $M = K$ (See \cite{zhang1}).

\begin{Lemma}\label{non}
If $H$ is a nontrivial weakly nilpotent hypergroup, then $Z(H) \neq 1$.
\end{Lemma}

\begin{proof}

Suppose that $Z(H) = 1$.
It implies from $Z_{2}(H)//Z(H) = Z(H//Z(H))$ that $Z_{2}(H) = 1$.
Similarly, $Z_{i}(H) = 1$ for all $i = 1, 2, \cdots$.
Then $H = Z_{n}(H) = 1$ for some integer $n$, a contradiction.
Therefore $Z(H) \neq 1$.
\end{proof}

\textbf{Proof of Theorem \ref{subnormal}}~~

We will take the following steps to prove the necessity of the theorem.

{\bf Step $(1)$} {\sl If $F$ is a proper closed subset of $H$, then $Z(H//F)$ is non-trivial.}

Since $H$ is a weakly nilpotent hypergroup, $H//F$ is a weakly nilpotent hypergroup by Theorem \ref{qu}.
Hence it follows from Lemma \ref{non} that $Z(H//F)$ is non-trivial.

{\bf Step $(2)$} {\sl If $M$ is a maximal closed subset of $H$, then $M$ is strongly normal in $H$.}

It follows from Step $(1)$ that $Z(H//M) \neq 1$.
Since $M$ is a maximal closed subset of $H$, $H//M$ has no non-trivial closed subsets by \cite[Theorem 3.4.6]{z4}.
Then $H//M = Z(H//M)$.
Hence $M$ is strongly normal in $H$.

{\bf Step $(3)$} {\sl For every closed subset $E$ of $H$, $E$ is strongly subnormal in $H$.}

In the case that $E=H$, clearly $E$ is strongly normal in $H$ and so $E$ is strongly subnormal in $H$.
Now assume that $E \neq H$ is a proper closed subset of $H$.
Then there exits a closed subset chain
$$E = E_{0} \leq E_{1} \leq \cdots \leq E_{n} = H$$
such that $E_{i}$ is a maximal closed subset of $E_{i+1}$ for $i=0, 1, \cdots, n-1$.
By Theorem \ref{sub}, $E_{i}$ is nilpotent for $i=0, 1, \cdots, n$.
It follows from  Step $(2)$ that $E_{i}$ is strongly normal in $E_{i+1}$ for $i=0, 1, \cdots, n-1$.
Therefore $E$ is strongly subnormal in $H$.
The necessity of the theorem holds on.
\qed
\medskip

\section{Proofs of Theorems \ref{solvable} and \ref{Sylow}}

\begin{Definition}

A finite residually thin hypergroup $P$ is called a $p$-hypergroup if $n_P$ is a power of $p$.

\end{Definition}

\begin{Lemma}\label{p-hypergroup}

If $H$ is a finite $RT$ $p$-hypergroup, then $H$ is a solvable hypergroup.

\end{Lemma}

\begin{proof}

Since $H$ is a $RT$ hypergroup, there exists a series of closed subsets of $H$:

$$1 = H_m \subseteq H_{m-1} \subseteq \cdots \subseteq H_2 \subseteq H_1 \subseteq H_0 = H$$
such that $H_i//H_{i+1}$ is  thin, $i= 0, 1, 2, \cdots, m-1$.
And  $|H_i//H_{i+1}|=n_{H_i//H_{i+1}}$ divides  $n_H$.
Then $|H_i//H_{i+1}|$ is a power of $p$.
Hence $H_i//H_{i+1}$ is a $p$-group, $i= 0, 1, 2, \cdots, m-1$.
So $H$ is a solvable hypergroup.
\end{proof}

\textbf{Proof of Theorem \ref{solvable}}~~

We are assuming that $H$ is a finite residually thin nilpotent hypergroup which is $p$-valenced.
Thus, by \cite[Theorem 1.2]{b2}, $H$ contains a strongly normal closed $p$-subset $F$ that contains all subnormal closed $p$-subsets of $H$.

Assume first that $F = \{ 1 \}$.
Then $\{ 1 \}$ is a is strongly normal in $H$, so that, by Lemma \ref{strong}, $H$ is  thin.
Thus, as $H$ is assumed to be weakly nilpotent, $H$ is a nilpotent group.
Hence $H$ is solvable.

Assume now that $F = H$. Then, $H$ is residually thin and $n_H$ a $p$-power.
Thus, by Lemma \ref{p-hypergroup}, $H$ is solvable.

Assume finally that $\{ 1 \} \neq F \neq H$.
Since $H$ is weakly nilpotent, $F$ and $H//F$ are both weakly nilpotent; cf. Theorem \ref{sub} and Theorem \ref{qu}, respectively.
As $F$ is a strongly normal closed $p$-subset of $H$,
$H//F$ is thin  by Lemma \ref{strong}.
Then $H//F$ is a nilpotent group.
Hence $H//F$ is solvable.
Since $H$ is residually thin,
$F$ is residually thin; cf. \cite[Theorem 6.2]{fz1}.
And since $H$ is $p$-valenced, $F$ is $p$-valenced; cf. Lemma \ref{a}.
Now recall that  $H$ is assume to be finite.
Thus, by \cite[Theorem 1.1]{b2}, $n_{F}n_{H//F} = n_H$.
Since $1 \neq F \neq H$, $n_H > n_{F}$.
Then induction yields that $F$ is solvable.
It follows that $H$ is a solvable hypergroup; cf. \cite[Lemma 5.3]{vz1}
\qed
\medskip

\begin{Lemma}\label{O}(See \cite[Theorem 1.2 and Theorem 1.3]{b2})
Let $p$ be a  prime number and $H$ a finite $RT$ hypergroup that is $p$-valenced.
Then

(1) $H$ contains a strongly normal $p$-subset $O_{p}(H)$ that contains all the subnormal $p$-subsets of $H$.

(2) Every Sylow $p$-subset of $H$ contains $O_{p}(H)$.
\end{Lemma}

\textbf{Proof of Theorem \ref{Sylow}}~~
By Lemma \ref{O}(1), we have that $O_{p}(H)$ is a strongly normal closed subset of $H$.
Then $H//O_{p}(H)$ is a thin hypergroup by Lemma \ref{strong}.
It implies from Theorem \ref{qu} that $H//O_{p}(H)$ is a weakly nilpotent hypergroup.
Hence $H//O_{p}(H)$ is nilpotent group.
By Lemma \ref{O}(2), $O_{p}(H) \subseteq P$ for every Sylow $p$-subset $P$ of $H$.
Then we obtain that $P//O_{p}(H)$ is a Sylow $p$-subgroup of $H//O_{p}(H)$.
By the Sylow Theorem of finite group, we have that $P//O_{p}(H)$ is a normal subgroup of $H//O_{p}(H)$,
that is $P//O_{p}(H)$ is a strongly normal closed subset of $H//O_{p}(H)$.
By Lemma \ref{normal}, $P$ is a strongly normal closed subset of $H$.
\qed
\medskip

\section{Some propositions and questions on weakly nilpotent hypergroups}

In the theory of finite groups, nilpotent groups play a very important role.
There are numerous research results on nilpotent groups (see monographs \cite{GuoI,Rob,H}).
The concept of weakly nilpotent hypergroups mentioned in this paper is actually a generalization of classical nilpotent groups.
In fact, nilpotent groups can be regarded as thin weakly nilpotent hypergroups and also thin weakly nilpotent hypergroups can be regarded as nilpotent groups.
In this section, we also get some characterizations of weakly niloptent hypergroups, which generalize some well-known results in finite nilpotent group theory.
And we also give some open questions about weakly nilpotent hypergroups.

\begin{Proposition}\label{Z}
Let $H$ be a finite hypergroup.
$H//Z(H)$ is a weakly nilpotent hypergroup if and only if $H$ is a weakly nilpotent hypergroup.
\end{Proposition}

\begin{proof}
Firstly, the sufficiency is clear by Theorem \ref{qu}.
Next we will show the necessity.
Assume that $Z(H) = 1$.
It holds on.
Assume that $Z(H) \neq 1$.
If $H//Z(H) = 1$, then $H = Z(H)$ and so $H$ is a weakly nilpotent hypergroup.
Suppose that $H//Z(H) \neq 1$.
Since $H//Z(H)$ is a weakly nilpotent hypergroup, $Z_{2}(H)//Z(H) = Z(H//Z(H)) \neq 1$ by Lemma \ref{non}.
Hence $1 \subsetneqq Z(H) \subsetneqq Z_{2}(H)$.
Similarly, we have that
$$1 \subsetneqq Z(H) \subsetneqq Z_{2}(H) \subsetneqq \cdots \subsetneqq Z_{m}(H) \subsetneqq \cdots.$$
Since $H$ is finite, there exists a positive integer $n$ such that $H = Z_{n}(H)$.
Therefore $H$ is a weakly nilpotent hypergroup.
\end{proof}

\begin{Proposition}\label{H}
Let $H$ be a finite hypergroup.
$H//Z_{\infty}(H)$ is a weakly nilpotent hypergroup if and only if $H$ is a weakly nilpotent hypergroup.
\end{Proposition}

\begin{proof}
Firstly, the sufficiency is clear by Theorem \ref{qu}.
Next we will show the necessity.
From the definition of the hypercenter of $H$, there exists a integer $n$ such that $Z_{\infty}(H) = Z_{n}(H)$
and $Z_{n}(H) = Z_{n+1}(H) = \cdots$.
If $H//Z_{\infty}(H) \neq 1$, then $Z(H//Z_{n}(H)) \neq 1$.
Hence $Z_{n+1}(H)//Z_{n}(H) \neq 1$ and so $Z_{n+1}(H)\neq Z_{n}(H)$, a contradiction.
Therefore $H = Z_{\infty}(H) = Z_{n}(H)$.
It implies that $H$ is a weakly nilpotent hypergroup.
\end{proof}

\begin{Remark}\label{z}

When $H$ is a group,
Propositions \ref{Z} and \ref{H} generalized some well-known results in the finite group theory.

\end{Remark}

\begin{Corollary}
(1)Let $H$ be a finite group.
$H//Z(H)$ is a nilpotent group if and only if $H$ is a nilpotent group.

(2) Let $H$ be a finite group.
$H//Z_{\infty}(H)$ is a nilpotent group if and only if $H$ is a nilpotent group.

\end{Corollary}

From Theorem \ref{subnormal}, we can get the following corollary.

\begin{Corollary}

Let $H$ be a finite nilpotent group.
Then every subgroup of $H$ is subnormal in $H$.

\end{Corollary}

\begin{Remark}

As we known, if every subgroup of a finite group $H$ is subnormal in $H$, then $H$ is nilpotent.
However, we do not know if this conclusion holds for hypergroups.

\end{Remark}

Then we have the following question.

\begin{Question}

Let $H$ be a finite hypergroup.
If any closed subset $F$ of $H$ is subnormal in $H$,
is $H$ a weakly nilpotent hypergroup $?$

\end{Question}

As well known, if every Sylow subgroup of a finite nilpotent group $G$ is normal in $G$, then $G$ is a nilpotent group.
However we do not know if  this conclusion holds for finite weakly nilpotent hypergroups.
Hence we have the following question.

\begin{Question}

Let $H$ be a finite weakly  nilpotent hypergroup.
If any Sylow closed subset $P$ of $H$ is normal in $H$,
is $H$ a nilpotent hypergroup $?$

\end{Question}

\noindent\textbf{Acknowledgements.} The authors would like to thank the referee for
his/her careful reading and many useful comments, which indeed improved the presentation
of this article.


\begin{thebibliography}{99}

\bibitem{b1}
H. Blau, P.-H. Zieschang, Sylow theory for table algebras, fusion rule algebras, and hypergroups, J.
Algebra \textbf{273} (2004) 551-570.

\bibitem{b2}
H. Blau, The $\pi$-radical and Hall's theorems for residually thin and $\pi$-valenced hypergroups, table algebras, and association schemes, J.
Algebra \textbf{600} (2022) 279-301.

\bibitem{fz1}
C. French, P.-H. Zieschang, On residually thin hypergroups, J. Algebra \textbf{551} (2020) 93-118.

\bibitem{fz2}
C. French, P.-H. Zieschang, Regular actions of Coxeter hypergroups, Comm. Algebra \textbf{52} (2024) 532-565.

\bibitem{GuoI}
W. Guo, The Theory of Classes of Groups, Science Press, Kluwer Academic Publishers, Beijing-New York-Dordrecht-Boston-London, 2000.

\bibitem{H}
B. Huppert, Endliche Gruppen I, Springer-Verlag, Berlin, 1967.

\bibitem{jun}
J. Jun, Association schemes and hypergroups, Comm. Algebra \textbf{46}(3) (2018) 942-960.

\bibitem{m1}
F. Marty, Sur une generalization de la notion de groupe, in: 8th Congress Math. Scandinaves, Stockholm, 1934, pp. 45-49.

\bibitem{Rob}
D. J. S. Robinson, A Course in the Theory of Groups, Springer-Verlag, New York, 1982.

\bibitem{tz1}
R. Tanaka, P.-H. Zieschang, On a class of wreath products of hypergroups and association schemes,
J. Algebraic Comb. \textbf{37} (2013) 601-619.

\bibitem{vz1}
A. Vasil'ev, P.-H. Zieschang, Solvable hypergroups and a generalization of Hall's theorems on solvable groups to association schemes, J. Algebra \textbf{594} (2022) 733-750.

\bibitem{zhang1}
C. Zhang, W. Guo, Nilponent hypergroups, Sci. China Math. (Chinese) \textbf{55}(9) (2025) 1689-1698.

\bibitem{z2}
P.-H. Zieschang, Hypergroups all non-identity elements of which are involutions, in: Advances in
Algebra 305-322, in: Springer Proc. Math. Stat., vol. 277, Springer, Cham, 2019.

\bibitem{z3}
P.-H. Zieschang, Hypergroups, Springer, Cham, 2023.

\bibitem{z4}
P.-H. Zieschang, Hypergroups, Max-Planck-Institut fur Mathematik Bonn, Preprint Series 97 (2010).


\end{thebibliography}
\end{document}